\documentclass[12pt,a4paper]{amsart}
\usepackage{amsthm,amssymb,verbatim,amsmath, amsfonts, amscd, color}
\usepackage[a4paper,left=25mm,right=25mm,top=30mm,bottom=30mm]{geometry}
\usepackage{enumitem}  
\usepackage{hyperref}  

\theoremstyle{plain}\newtheorem{Theorem}{Theorem}[section]
\theoremstyle{plain}
\theoremstyle{plain}\newtheorem{Corollary}[Theorem]{Corollary}
\theoremstyle{plain}\newtheorem{Lemma}[Theorem]{Lemma}
\theoremstyle{plain}\newtheorem{Proposition}[Theorem]{Proposition}
\theoremstyle{definition}
\theoremstyle{definition}
\theoremstyle{definition}
\theoremstyle{definition}
\theoremstyle{definition}
\theoremstyle{definition}
\theoremstyle{definition}
\theoremstyle{definition}
\theoremstyle{definition}
\theoremstyle{definition}
\theoremstyle{definition}
\theoremstyle{definition}
\theoremstyle{definition}\newtheorem{Notation/Definition}
[Theorem]{Notation/Definition}
\theoremstyle{definition}
\theoremstyle{plain}

 \def\Syl{\mathrm{Syl}}

\newcommand{\PSL}{\operatorname{PSL}}
\newcommand{\SL}{\operatorname{SL}}
\newcommand{\PGL}{\operatorname{PGL}}
\newcommand{\GL}{\operatorname{GL}}
\newcommand{\PSU}{\operatorname{PSU}}
\newcommand{\SU}{\operatorname{SU}}

\newcommand{\Out}{{\text{Out}}}

\begin{document}


\title[Principal blocks with tame representation type]{Splendid Morita equivalences for principal blocks with 
semidihedral defect groups}

\date{\today}
\author{Shigeo Koshitani, Caroline Lassueur and Benjamin Sambale}
\address{Center for Frontier Science,
Chiba University, 1-33 Yayoi-cho, Inage-ku, Chiba 263-8522, Japan.}
\email{koshitan@math.s.chiba-u.ac.jp}
\address{FB Mathematik, TU Kaiserslautern, Postfach 3049, 67653 Kaiserslautern, Germany.}
\email{lassueur@mathematik.uni-kl.de}
\address{Institut f{\"u}r Algebra, Zahlentheorie und Diskrete Mathematik, 
Leibniz Universit{\"a}t Hannover, Welfengarten 1, 30167 Hannover, Germany.}
\email{sambale@math.uni-hannover.de}

\thanks{The first author was partially supported by the Japan Society for
Promotion of Science (JSPS), Grant-in-Aid for Scientific Research
(C)19K03416, 2019--2021.
The second author acknowledges financial support by DFG SFB/TRR 195.
The third author is supported by the DFG grants \mbox{SA 2864/1-2} and \mbox{SA 2864/3-1}.}



\dedicatory{Dedicated to Gunter Malle on his 60th Birthday.}

\begin{abstract}
We classify principal blocks of finite groups with semidihedral defect groups up to splendid Morita equivalence.  This completes the classification of all principal $2$-blocks of tame representation type up to splendid Morita equivalence and  shows that Puig's Finiteness Conjecture holds for such blocks.
\end{abstract}

\maketitle


\section{Introduction}
The present article is motivated by Puig's Finiteness Conjecture (see  \cite[(38.6) Conjecture]{The95}), strengthening Donovan's Conjecture and  
predicting that for a given prime number~$\ell$ and a finite $\ell$-group $P$ there are only finitely many isomorphism classes of interior $P$-algebras arising as source algebras of $\ell$-blocks of finite groups with defect groups isomorphic to $P$, or equivalently that there are only a finite number of {\it splendid Morita equivalence}  classes of blocks of finite groups with defect groups isomorphic to~$P$. 
The cases where $P$ is either cyclic \cite{Lin96b} or a Klein-four group \cite{CEKL11} are the only cases for which this conjecture has been proved to hold in full generality. Else, under additional assumptions, Puig's Finiteness Conjecture has also been proved for several classes of finite groups, as for instance for $\ell$-soluble groups \cite{Pui94}, for the symmetric groups \cite{Pui94}, for the alternating groups and  the double covers thereof,  for Weyl groups, or for classical groups, see \cite{HK00, HK05} and the references therein.
\par
Our principal aim in this article is to classify principal $2$-blocks of finite groups with semidihedral defect groups up to splendid Morita equivalence and deduce that Puig's Finiteness Conjecture holds when letting the blocks vary through the class of all principal $2$-blocks of tame representation type.  We show that the knowledge of the equivalence classes of principal blocks with dihedral defect groups up to splendid Morita equivalences
is enough to describe  
the splendid Morita equivalence classes of  principal blocks with semidihedral defect groups, as well as the bimodules realizing these equivalences.
We recall that Erdmann \cite{Erd90} classified blocks of tame representation type up to Morita equivalence by describing their  basic algebras by generators and relations making intense use of the Auslander-Reiten quiver. However, given a splitting $2$-modular system $(K,\mathcal O, k)$, her results are not liftable to $\mathcal O$ in general and do not imply that the resulting Morita equivalences are necessarily splendid Morita equivalences.
By contrast, if Puig's Finiteness Conjecture holds over $k$, then it automatically holds over~$\mathcal O$, since the bimodules inducing splendid Morita equivalences are liftable from $k$ to $\mathcal O$. 
\par
To state our main results, we introduce the following notation. 
For $m\ge 2$ and a prime power $q=p^f$ let 
\begin{align*}
\SL^{\pm1}_m(q):=\{A\in\GL_m(q)\,|\,\det(A)=\pm1\},\\
\SU^{\pm1}_m(q):=\{A\in\mathrm{GU}_m(q)\,|\,\det(A)=\pm1\}.
\end{align*}
Now let $q=p^{2f}$ where $p$ is an odd prime. Then there are exactly three groups $H$ with 
\[\PSL_2(q) < H < \mathrm{P\Gamma L}_2(q)
 \text{ \quad and \quad} |H:\PSL_2(q)|=2. 
 \]
 One is $\PGL_2(q)$, one is contained in $\PSL_2(q)\rtimes\langle F\rangle$ where $F$ is the Frobenius automorphism on 
 $\mathbb F_q$, and the third one is denoted by 
$\PGL^*_2(q)$ (see \cite{G}).
Our main result is as follows:

\begin{Theorem}\label{MainResult}
Let $G$ be a finite group with a semidihedral
Sylow $2$-subgroup $P$ of order $2^n$ with $n\geq 4$ fixed and let $k$ be an algebraically closed field of characteristic~$2$.
Then the following assertions hold.
\begin{enumerate}
\item[\rm (a)] The principal block $B_0(kG)$ of $kG$  is splendidly Morita equivalent to the principal block of precisely one of the following groups:
\begin{itemize}
\item[\sf (bb)] $P$;
\item[\sf (ba1)] $\SL^{\pm}_2(p^f)$ where $4(p^f+1)_2=2^n$;
\item[\sf (ba2)] $\SU^{\pm}_2(p^f)$ where $4(p^f-1)_2=2^n$;
\item[\sf (ab)] $\PGL^*_2(p^{2f})$ where $2(p^{2f}-1)_2=2^n$;
\item[\sf (aa1)] $\PSL_3(p^f)$ where $4(p^f+1)_2=2^n$; or
\item[\sf (aa2)] $\PSU_3(p^f)$ where $4(p^f-1)_2=2^n$;
\end{itemize}
where $p$ is an odd prime number and $f\geq 1$.
Moreover, the splendid Morita equivalence is realized by the Scott module ${{\mathrm{Sc}}(G\times G', \Delta P)}$, where $G'$ is the group listed in the corresponding case.
\item[\rm(b)] In particular, if $G$ and $G'$ are two groups such that $|G|_2=|G'|_2$ and  which are both of type {\sf(ba1)}, both of type {\sf(ba2)},  both of type {\sf(ab)}, both of type {\sf(aa1)}, or both of type {\sf(aa2)},  then $B_0(kG)$ and $B_0(kG')$ are splendidly Morita equivalent.
\end{enumerate}
\end{Theorem}
In part (a) the labeling of the fusion patterns originates from \cite[p.231]{Ols75} (see also \cite[Theorem 5.3]{CG12}) and we emphasize that  $G'$ is not the derived subgroup $[G,G]$.
\medskip

Finally, as Craven-Eaton-Kessar-Linckelmann  proved in \cite{CEKL11} that Puig's Finiteness Conjecture holds for $2$-blocks with Klein-four defect groups and the first and the second authors proved in \cite{KL20a, KL20b} that it holds as well for principal $2$-blocks with dihedral and generalized quaternion defect groups, Theorem~\ref{MainResult} yields the following 
{\color{black}corollary}:

\begin{Corollary}
Puig's Finiteness Conjecture holds for principal $2$-blocks of tame representation type.
\end{Corollary}


\section{Notation}

Throughout this paper, unless otherwise stated we adopt the following notation and 
conventions.  
All groups considered  are 
finite and all modules are finitely generated right modules.
In particular $G$ always denotes a finite group.  
We denote the dihedral group of order $2^m$ with $m\geq 2$
by $D_{2^m}$, the generalized quaternion group of order $2^m$ with
$m\geq 3$ by $Q_{2^m}$, and the cyclic group of order $m\geq 1$ by $C_m$.  
We denote by $SD_{2^n}$ 
the semidihedral group of order $2^n$ where $n \geq 4$ is fixed throughout this paper.
We denote by $\mathcal F_P(G)$ the fusion system of $G$ on
a Sylow $p$-subgroup $P$ of $G$ and 
by ${\mathrm{Syl}}_p(G)$ the set of all
Sylow $p$-subgroups of $G$.
We write $\Delta G:=\{(g, g)\,|\, g\in G\} \leq G\times G$.
Given two subgroups $N \vartriangleleft G$ and $L\leq G$ with $G=NL$ and $N\cap L=1$, $N\rtimes L$ denotes the semi-direct product of $N$ by $L$. 
\par
We let $k$ be an algebraically closed field of characteristic~$2$. We write $B_0(kG)$ for the principal block of the group algebra $kG$. For a block $B$ of $kG$, we write  $1_B$ for the block idempotent of $B$, $C_B$ for the Cartan matrix of $B$, and $k(B)$ and $\ell(B)$, respectively, are the numbers of irreducible
ordinary and Brauer characters of $G$ belonging to $B$. {\color{black} We denote by ${\mathrm{mod}}\text{-}B$ the category of finitely generated right $B$-modules and by $\underline{\mathrm{mod}}\text{-}B$ the associated stable module category.} 
\par
We denote by 
$k_G$ the trivial $kG$-module. {\color{black} Given a $kG$-module $M$ and a $2$-subgroup $Q\leq G$ we denote by $M(Q)$ the Brauer construction of $M$ with respect to $Q$.}
For $H\leq G$ we denote by ${\mathrm{Sc}}(G,H)$ the Scott $kG$-module with respect to~$H$.
By definition ${\mathrm{Sc}}(G,H)$ is, up to isomorphism, the unique indecomposable direct summand
of the induced module ${k_H}{\uparrow}^G$ which contains $k_G$ in its
head (or equivalently in its socle) and is a $2$-permutation module by definition.  See
\cite[Chap.4 \S 8.4]{NT89}. 
Equivalently, ${\mathrm{Sc}}(G,H)$ is the relative $H$-projective cover of $k_G$ (see \cite[Proposition 3.1]{The85}). 
\par
If $G$ and $H$ are finite groups,  $A$ and $B$ are blocks of $kG$ and $kH$ respectively and $M$ is an 
$(A,B)$-bimodule  inducing a Morita equivalence between $A$ and $B$, then we view $M$ as a right $k[G\times H]$-module via the right $G\times H$-action defined by  $m\cdot (g,h):=g^{-1}mh$ for every $m\in M, g\in G, h\in H$.
Furthermore, the algebras $A$ and $B$ are called 
\emph{splendidly Morita equivalent} (or \emph{source-algebra equivalent}, or  \emph{Puig equivalent}), if  
there is a Morita equivalence between $A$ and $B$ induced by an 
$(A,B)$-bimodule $M$ such that $M$, viewed  as a right $k[G\times H]$-module,
is a $2$-permutation module. 
In this case, due to a result of Puig (see \cite[Corollary~7.4]{Pui99} and 
\cite[Proposition~9.7.1]{Lin18}), the defect groups $P$ and $Q$ of $A$ and
$B$ respectively are isomorphic. Hence from now on we identify $P$ 
and $Q$.
Obviously $M$ is indecomposable as a $k(G\times H)$-module and since $M$ induces a Morita equivalence, $_AM$ and $M_B$ are both projective and therefore $\Delta P \leq G\times H$ is a vertex of $M$.
By a result of Puig and Scott, this definition  is equivalent to the condition that $A$ and $B$ have source algebras which are isomorphic as interior $P$-algebras (see \cite[Theorem~4.1]{Lin01} and \cite[Remark 7.5]{Pui99}).
\par
In this paper, in order to produce splendid Morita equivalences between principal
blocks of two finite groups $G$ and $G'$ with a common defect group $P$, we will use $2$-permutation modules  given by Scott modules of the form 
${\mathrm{Sc}}(G\times G', \, \Delta P)$, which are obviously $(B_0(kG),B_0(kG'))$-bimodules.   
Furthermore, we shall rely on the classification of  principal $2$-blocks 
of finite groups with dihedral  Sylow $2$-subgroups, up to splendid 
Morita equivalence, obtained  in \cite{KL20a}, where the result for Klein-four groups is 
in \cite{CEKL11}.
We will use the results of \cite{CEKL11,KL20a,KL20b} without further introduction in this text and refer the reader directly to the relevant material in these articles.

\section{Finite groups with semidihedral Sylow 2-subgroups}

One of the starting points of this project is the following very useful observation
due to the third author:

\begin{Theorem}[\cite{ABG70}]\label{thm:semidihderal}
Let $G$ be a finite group with a semidihedral
Sylow $2$-subgroup $P$ of order $2^n$ with $n\geq 4$,
and assume that $O_{2'}(G)=1$. 
Then one of the following holds:
\begin{enumerate}
\item[\rm (bb)] \ $G=P$.
\item[\rm (ba1)] \ $G={\mathrm{SL}}^{\pm}_2(p^f)\rtimes C_d$ 
where $4(p^f+1)_2=2^n$ and $d$ is an odd divisor of $f$.
\item[\rm(ba2)] \ $G={\mathrm{SU}}^{\pm}_2(p^f)\rtimes C_d$ 
where $4(p^f-1)_2=2^n$ and $d$ is an odd divisor of $f$.
\item[\rm (ab)] \ $G={\mathrm{PGL_2^*}}(p^{2f})\rtimes C_d$
where $2(p^{2f}-1)_2=2^n$ and $d$ is an odd divisor of $f$.
\item[\rm (aa1)] \ $G={\mathrm{PSL}}_3(p^f).H$ 
where $4(p^f+1)_2=2^n$ and $H \leq C_{(3, p^f-1)}\times C_d$
for an odd divisor $d$ of $f$.
\item[\rm (aa2)] \ $G={\mathrm{PSU}}_3(p^f).H$ where 
$4(p^f-1)_2=2^n$ and $H \leq C_{(3,p^f+1)}\times C_d$
for an odd divisor $d$ of $f$.
\item[\rm (aa3)] \ $G=M_{11}$.
\end{enumerate}
\end{Theorem}
\begin{proof}
If $G$ is $2$-nilpotent, then Case (bb) holds since $O_{2'}(G)=1$. 
In all other cases, $G$ is a $D$-group, a $Q$-group or a $QD$-group with the notation of \cite[Definition~2.1]{ABG70}.
If $G$ is a $D$-group, then $G$ has a normal subgroup $K$ of index $2$ with a dihedral Sylow $2$-subgroup. Hence, the structure of $K$ (and $G$) follows essentially from the classification of Gorenstein--Walter. The precise information can be extracted from 
Proposition~3.4 of \cite{ABG70} and its proof. We see that Case (ab) holds. 
If $G$ is a $Q$-group, then Case (ba1) or (ba2) occurs by Propositions 3.2 and 3.3 (and its proof) of \cite{ABG70}. Finally, let $G$ be a $QD$-group. Then by \cite[Proposition 2.2]{ABG70}, $N:=O^{2'}(G)$ is simple and the possible isomorphism types of $N$ are given in the third main theorem of \cite{ABG70}, namely $M_{11}$, $\PSL_3(p^f)$ and $\PSU_3(p^f)$. Since $C_G(N)\cap N=Z(N)=1$ we have $C_G(N)\le O_{2'}(G)=1$. The possibilities for $G/N\le\Out(N)$ can be deduced from the Atlas~\cite{Atlas}. In particular, Case (aa1) holds if $N\cong M_{11}$. Now let $N$ be $\PSL_3(p^f)$ or $\PSU_3(p^f)$. Since, $G/N$ has odd order, it does not induce graph automorphisms on $N$. Hence, $G/N\le C_{(3,p^f-1)}\rtimes C_f$ or $G/N\le C_{(3,p^f+1)}\rtimes C_f$. Again, since $G/N$ has odd order, $G/N$ is abelian. 
\end{proof}

\section{The principal $2$-blocks of $M_{11}$ and  ${\mathrm{PSL}}_3(3)$}

Benson and Carlson \cite[(14.1)]{BC87} observed that
the principal $2$-blocks of the groups ${\mathrm{PSL}}_3(3)$ and $M_{11}$ 
are Morita equivalent by comparing their basic algebras.
In this section, we prove that  their result can be refined to a splendid Morita equivalence. 
More precisely, we prove that this Morita equivalence is induced by a Scott module 
using the gluing method developed by the first and the second author 
in \cite[Section~3 and Section 4]{KL20a}.

\begin{Lemma}\label{stableEquiv}
Set $G:={\mathrm{PSL}}_3(3)$, $G':=M_{11}$ 
and let $P\in{\mathrm{Syl}}_2(G)\cap{\mathrm{Syl}}_2(G')$,
so that $P\cong SD_{16}$.
Then
${\mathrm{Sc}}(G\times G',\, \Delta P)$ induces a stable equivalence of
Morita type between $B_0(kG)$ and $B_0(kG')$.
\end{Lemma}
\begin{proof} 
Set $P:=\langle s, t \,|\, s^8=t^2=1, \, tst =s^{3} \,\rangle\cong SD_{16}$,  
$z:= s^4$ and $Z:=\langle z\rangle=Z(P)\cong C_2$ and observe that $\mathcal F_P(G)=\mathcal F_P(G')$ by \cite[Theorem 5.3]{CG12}.
We read from the Atlas \cite[p.13 and p.18]{Atlas} that $C_G(z)\cong  {\mathrm{GL}}_2(3)\cong C_{G'}(z)$.
Note that $k{\mathrm{GL}}_2(3)$ has only one $2$-block,
namely the principal block since $O_{2'}({\mathrm{GL}}_2(3))=1$.
Thus, 
${\mathrm{Sc}}(C_G(z)\times C_{G'}(z), \Delta P)$
realizes a (splendid) Morita equivalence between
$B_0(k\,C_G(z))$ and $B_0(k\, C_{G'}(z))$ {\color{black}because ${\mathrm{Sc}}(C_G(z)\times C_{G'}(z), \Delta P)= kC_G(z)$ seen as $(kC_G(z),kC_{G'}(z))$-bimodule.}
On the other hand,
${\mathrm{Sc}}(C_G(z)\times C_{G'}(z), \Delta P)\,|\,M(\Delta Z)$ by \cite[Lemma 3.2]{KL20a}.
Hence{\color{black}, as  ${\mathrm{Sc}}(C_G(z)\times C_{G'}(z), \Delta P)$ is Brauer indecomposable by \cite[Theorem~1.2]{KT19}}, we have in fact 
${\mathrm{Sc}}(C_G(z)\times C_{G'}(z), \Delta P)\cong M(\Delta Z)$. 
Thus, $M(\Delta Z)$ induces a Morita
equivalence between $B_0(k\,C_G(z))$ and $B_0(k\,C_{G'}(z))$. 
Therefore, as all involutions in $G$ 
are $G$-conjugate and $\mathcal F_P(G)=\mathcal F_P(G')$,  
\cite[proof of Case 1 of Proposition 4.6]{KL20a} yields that  for every involution $t\in P$
$${\mathrm{Sc}}( C_G(t) \times C_{G'}(t),\Delta P) = M(\Delta \langle t\rangle)$$
and induces a Morita equivalence between  $B_0(kC_{G}(t))$
and $B_0(kC_{G'}(t))$. Therefore the assertion follows from
\cite[Lemma 4.1]{KL20a}. 
\end{proof}

\begin{Proposition}\label{M11-L3(3)}
With the notation of Lemma~\ref{stableEquiv}, ${\mathrm{Sc}}(G\times G',\,\Delta P)$ induces a splendid Morita equivalence between $B_0(kG)$ and $B_0(kG')$.
\end{Proposition}

\begin{proof}
Set $M:={\mathrm{Sc}}(G\times G',\,\Delta P)$, $B:=B_0(kG)$ and $B':=B_0(kG')$. The block $B$ has three simple $kG$-modules: $k_G$ and two modules ${\sf 12}$ and ${\sf 26}$ of $k$-dimension $12$ and $26$ respectively.
Similarly the block $B':=B_0(kG')$ has
three simple $kG'$-modules:  $k_{G'}$ and two modules ${\sf 10}$ and  ${\sf 44}$ of $k$-dimension $12$ and $26$ respectively. (See \cite{ModAtlas}).
\par
To start with, we claim that these six simple modules are all trivial source modules.
First, the trivial modules $k_G$ and $k_{G'}$ are obviously trivial source modules with vertex $P$, and  for $G':=M_{11}$, the module $\sf 10$ is a trivial source module with vertex $Q_8$ by  \cite[Lemma~2.1(a) and (d)]{SchM12}, whereas the module $\sf 44$ is a trivial source module with vertex $C_2\times C_2$ by  \cite[Lemma~2.2(a) and (c)]{SchM12}.
Next, consider $G:={\mathrm{PSL}}_3(3)$ and its maximal subgroup 
 $\mathfrak M:=3^2 \rtimes 2\mathfrak S_4 = 3^2 \rtimes {\mathrm{GL}}_2(3)$
where $\mathfrak S_4$ is the
symmetric group of degree $4$ (see \cite[p.13]{Atlas}).
Using the Atlas \cite[p.13]{Atlas} and the $2$-decomposition matrix
of $B$ given in \cite{ModAtlas} we easily compute that 
$k_{\mathfrak M}{\uparrow}^G=k_G + {\sf 12}$ as composition factors. 
Then, as $k_{G}$ and $\sf 12$ are self-dual, we must have  $k_{\mathfrak M}{\uparrow}^G=k_G \oplus {\sf 12}$. Hence {\sf 12} is a trivial source module.
Moreover, the module $\sf 12$ is liftable and affords the ordinary character $\chi_2$ (in the Atlas notation \cite[p.13]{Atlas}). Therefore, it follows from \cite[II Lemma 12.6(ii)]{Lan83} and  the character values of $\chi_2$ at $2$-elements that $\sf 12$ has vertex $C_2\times C_2$. 
To prove that $\sf 26$ is a trivial source module, we consider  $SD_{16}=P < {\mathrm{GL}}_2(3)=:\tilde G <\mathfrak M < G$.
We easily compute that $1_P{\uparrow}^{\tilde G}=1_{\tilde G}+\tilde\chi_{2a}$
where $\tilde\chi_{2a}$ is the unique $2$-rational irreducible ordinary character of $\tilde G$
of degree $2$. Hence, as above by self-duality, $k_P{\uparrow}^{\tilde G}=k_{\tilde G}\oplus {\tilde 2}$
where $\tilde 2$ is the unique simple $k\tilde G$-module in
$B_0(k \tilde G)$, so that the simple module $\tilde 2$ is a trivial source $k\tilde G$-module.
Again, we read from the the character table of $\tilde G$ and \cite[II Lemma 12.6(ii)]{Lan83} that $\tilde 2$ has vertex $Q_8$. 
Moreover, by the character tables of $G$ and $\tilde G$,  
we have $\tilde\chi_{2a}{\uparrow}^G = \chi_8$, so that ${\tilde 2}{\uparrow}^G = {\sf 26}$.
Hence {\sf 26} is also a trivial source $kG$-module with vertex $Q_8$.
\par
Next, we recall that there is a bijection between
the set of isomorphism classes of indecomposable  trivial source $kG'$-modules (resp. $kG$-modules) with
vertex $X\leq P$ and the set of 
isomorphism classes of indecomposable projective $k[N_{G'}(X)/X]$-modules (resp. $k[N_{G}(X)/X]$-modules).
(See \cite[Chap.4, Problem 10]{NT89}). 
Now consider $Q\leq P$ with $Q\cong Q_8$ and $K\leq P$ with $K\cong C_2\times C_2$.
It is easy to compute (e.g. using {\sf GAP}) that  $N_{G'}(P)/P=1$ 
and $N_{G'}(Q)/Q\cong N_{G'}(K)/K \cong \mathfrak S_3$ and it is well-known that $k \mathfrak S_3$ has two PIMs.
Hence there are  precisely two non-isomorphic indecomposable trivial source $kG'$-modules with vertex $Q$.
One of them is {\sf 10} by the above, and the other one has to be  ${\mathrm{Sc}}(G',Q)$, 
since ${\mathrm{Sc}}(G',Q)\ncong {\sf 10}$ as it must contain a copy of the trivial module in its socle. Namely, 
\begin{equation}\label{tsmM11VertexQ}
\{ \text{iso. classes of indec. trivial source $B'$-modules with vertex }Q \}=\{{\mathrm{Sc}}({G'},Q), {\sf 10}\}
\end{equation}
and similarly 
\begin{equation}\label{tsmM11VertexK}
\{\text{iso. classes of indec. trivial source $B'$-modules with vertex }K\}=\{{\mathrm{Sc}}({G'},K), {\sf 44}\}\,.
\end{equation}
For $G$ we also have $N_{G}(P)/P=1$ and  $N_{G}(Q)/Q \cong N_{G}(K)/K \cong \mathfrak S_3$ 
(e.g. using~{\sf GAP}).
Thus, the same arguments as above yield:
\begin{equation}\label{tsmSL3-3VertexQ}
\{ \text{iso. classes of indec. trivial source $B$-modules with vertex }Q \}=\{{\mathrm{Sc}}({G},Q), {\sf 26}\}
\end{equation}
and 
\begin{equation}\label{tsmSL3-3VertexK}
\{ \text{iso. classes of indec. trivial source $B$-modules with vertex }K \}=\{{\mathrm{Sc}}({G},K), {\sf 12}\}\,.
\end{equation}
Now, let us consider the functor
$$F: 
{\mathrm{mod}}\text{-}B \rightarrow
{\mathrm{mod}}\text{-}B',
\ \ \  X_B \mapsto (X\otimes_B M)_{B'}.
$$
By Lemma \ref{stableEquiv}, $F$ is a functor realizing {\color{black} a stable equivalence of Morita type, hence an} additive category equivalence between
$\underline{\mathrm{mod}}\text{-}B$ and
$\underline{\mathrm{mod}}\text{-}B'$.
Therefore, as $\mathcal F_P(G)=\mathcal F_P(G')$ (see \cite[Theorem 5.3]{CG12}), first  by \cite[Lemma~3.4(a)]{KL20a} we have 
$$F(k_G)=k_{G'}\,,$$
and by \cite[Theorem~2.1(a)]{KL20a}, 
$F({\sf 26})$ and $F({\sf 10})$ are both indecomposable
$kG'$-modules in $B'$. 
Next, we prove that $F({\sf 26})= {\sf 10}$.
 {\color{black} It follows from \cite[Lemma 3.4(b)]{KL20a}} that  
$$F({\sf 26}) \in \{ {\mathrm{Sc}}(G',Q), \ {\sf 10}  \}. $$
If $F({\sf 26})={\mathrm{Sc}}(G',Q)$, then
\begin{align*}
0\,{\not=}&\, {\mathrm{Hom}}_{kG'}(F({\sf 26}), k_{G'})
\,=\, {\mathrm{Hom}}_{kG'}( {\sf 26}\otimes_{kG}M, k_{G'})
\\
=\, &
\,{\mathrm{Hom}}_{kG}({\sf 26}, \,  k_{G'}\otimes_{kG'}M^*)
\ \  \text{ by adjointness}
\\
=\,&\,
{\mathrm{Hom}}_{kG}( {\sf 26}, k_{G}) \ \ \ 
\text{by \cite[Lemma 3.4(a)]{KL20a}}
\\
=\,&\, 0,
\end{align*}
a contradiction, so that we have  $F({\sf 26})={\sf 10}$. A similar argument using {\rm(2)} and {\rm(4)} yields $F({\sf 12})={\sf 44}$.
Therefore, by
\cite[Theorem 4.14.10]{Lin18}, $F$, namely $M$, induces a Morita equivalence between $B$ and $B'$ {\color{black}because all simple $B$-modules are mapped to simple $B'$-modules.} 
\end{proof}

\section{Proof of Theorem \ref{MainResult}{\rm(b)}} \label{sec:proofMainb}

First of all, we give a lemma which is a direct consequence of a well-known result due to Alperin and Dade, {\color{black} see \cite{Alp76} and \cite{Dad77},  restated in terms of splendid Morita equivalences in  \cite[Theorem 2.2]{KL20a}.}

\begin{Lemma}\label{lem:AlperinDade}
\color{black}Assume $k$ is an algebraically closed field of arbitrary prime characteristic~$\ell$. 
Let $G$ and $G'$ be finite groups with a
common Sylow $\ell$-subgroup 
$P\in{\mathrm{Syl}}_{\ell}(G)\cap{\mathrm{Syl}}_{\ell}(G')$.
Assume further that there are finite groups $\widetilde G$ and
$\tilde G'$ such that $\widetilde G \vartriangleright G$ and $\widetilde G' \vartriangleright G'$, $\widetilde G/G$ and $\widetilde G'/G'$ are $\ell'$-groups, 
and $\widetilde G=C_G(P)\,G$, $\widetilde G'=C_{G'}(P)\,G'$. 
If ${\mathrm{Sc}}(G\times G', \Delta P)$ realizes a 
(splendid) Morita equivalence between $B_0(kG)$ and $B_0(kG')$,
then
${\mathrm{Sc}}(\widetilde G\times \widetilde G', \Delta P)$
realizes a (splendid) Morita equivalence between
$B_0(k\widetilde G)$ and $\widetilde B_0(k\widetilde G')$.
\end{Lemma}

\begin{proof}
Set $B:=B_0(kG)$, $\widetilde B:=B_0(k\widetilde G)$,
$B':=B_0(kG')$ and $\widetilde B':=B_0(k\widetilde G')$. 
By \cite[Theorem 2.2]{KL20a}, 
$\widetilde B$ and $B$
are splendidly Morita equivalent via
$1_{\widetilde B} k\widetilde G 1_B = 
{\mathrm{Sc}}(\widetilde G\times G, \Delta P)$,
and 
$B'$ and $\widetilde B'$
are splendidly Morita equivalent via
$1_{B' }k{\widetilde G'}\,1_{\widetilde B'} = 
{\mathrm{Sc}}(G' \times \widetilde G', \Delta P)$.
Furthermore ${\mathrm{Sc}}(G\times G',\Delta P)$
induces a splendid Morita equivalence between $B$ and $B'$ by assumption.
Hence composing these three splendid Morita equivalences, we have that
$$
{\mathrm{Sc}}(\widetilde G\times G, \Delta P)
\otimes_B {\mathrm{Sc}}(G\times G', \Delta P)
\otimes_{B'}{\mathrm{Sc}}(G' \times \widetilde G', \Delta P){\color{black} =:\mathfrak M}
$$
induces a splendid Morita equivalence between $\widetilde B$ and $\widetilde B'$.  
It  remains to see that 
$\mathfrak M={ {\mathrm{Sc}}(\widetilde G\times\widetilde G', \Delta P)}$.
Indeed, by the above
\begin{equation*}
\begin{split}
 \mathfrak M    
 \, =  \,  1_{\widetilde B} k\widetilde G 1_B
\otimes_B {\mathrm{Sc}}(G\times G', \Delta P)
\otimes_{B'}1_{B' }k{\widetilde G'}\,1_{\widetilde B'}    
& \, \Big| \, k\widetilde G\otimes_{kG}(kG\otimes_{kP}kG')
                                 \otimes_{kG'}k\widetilde{G'}    \\
        &  \,= \, k\widetilde G\otimes_{kP} k\widetilde{G'}
           = \ k_{\Delta P}{\uparrow}^{\widetilde G\times\widetilde{G'}}                            
\end{split}
\end{equation*}
and hence the definition of the Scott module yields ${\mathrm{Sc}}(\widetilde G\times \widetilde G', \Delta P)
\ \big| \,  \mathfrak{M}$. 
However, as $\mathfrak M$ induces a Morita equivalence between
$\widetilde B$ and $\widetilde B'$, which are both indecomposable
as bimodules, $_{\widetilde B}\mathfrak M_{\widetilde B'}$
must be indecomposable and it follows that $\mathfrak M
={\mathrm{Sc}}(\widetilde G\times\widetilde G', \Delta P)$.
\end{proof}

We now prove Theorem~\ref{MainResult}\rm(b) through a case-by-case analysis.
For the remainder of this section, we let $p,p'$ be prime numbers, $f,f'\geq 1$ and $n\geq 4$ be integers, and we use, without further mention, the fact that if a Scott module induces Morita equivalence between two blocks, then this equivalence is automatically splendid.

\begin{Proposition}\label{thm:SL2^pm}
Let  $G:={\mathrm{SL}}_2^\pm(p^f)$ and $G':={\mathrm{SL}}_2^\pm(p'^{f'})$ with $4(p^f+1)_2=4(p'^{f'}+1)_2=2^n$ 
and let $P\in{\mathrm{Syl}}_2(G)\cap{\mathrm{Syl}}_2(G')$.
Then, ${\mathrm{Sc}}(G\times G',\,\Delta P)$ induces a splendid 
Morita equivalence between $B_0(kG)$ and $B_0(kG')$.
\end{Proposition}

\begin{proof}
The groups $G$ and $G'$ have a common central subgroup $Z\leq P$ of order $2$ such that
$\bar G:= G/Z \cong {\mathrm{PGL}}_2(p^f)$ and $\bar{G'}:= G'/Z \cong {\mathrm{PGL}}_2(p'^{f'})$ have a common Sylow $2$-subgroup $\bar P:=P/Z$ isomorphic to $D_{2^{n-1}}$
(see \cite[p.4]{ABG70}).  
Hence it follows from \cite[Theorem~1.1(6)]{KL20a} that ${\mathrm{Sc}}({\bar G}\times {\bar G'},\,\Delta {\bar P})$ induces a splendid 
Morita equivalence between $B_0(k\bar G)$ and $B_0(k\bar G')$. Therefore ${\mathrm{Sc}}(G\times G',\,\Delta P)$ induces a   
Morita equivalence between $B_0(kG)$ and $B_0(kG')$ \medskip by \cite[Proposition~3.3(b)]{KL20b}. The claim follows.
\end{proof}

\begin{Proposition}\label{thm:SU2^pm}
Let $G:={\mathrm{SU}}_2^\pm(p^f)$ and $G':={\mathrm{SU}}_2^\pm(p'^{f'})$ with $4(p^f-1)_2=4(p'^{f'}-1)_2=2^n$
and let $P\in{\mathrm{Syl}}_2(G)\cap{\mathrm{Syl}}_2(G')$.
Then, ${\mathrm{Sc}}(G\times G',\,\Delta P)$ induces 
a splendid Morita equivalence between $B_0(kG)$ and $B_0(kG')$.
\end{Proposition}

\begin{proof}
Again, the groups $G$ and $G'$ have a common central subgroup $Z\leq P$ of order $2$ such that
$\bar G:= G/Z \cong {\mathrm{PGL}}_2(p^f)$ and $\bar{G'}:= G'/Z \cong {\mathrm{PGL}}_2(p'^{f'})$ have a common Sylow $2$-subgroup $\bar P:=P/Z$ isomorphic to $D_{2^{n-1}}$
(see \cite[p.4]{ABG70}). Hence the assertion follows from the same argument as in the proof of Proposition~\ref{thm:SL2^pm}, where \cite[Theorem~1.1(6)]{KL20a} is replaced by \cite[Theorem~1.1(5)]{KL20a}.
\end{proof}

\begin{Proposition}\label{thm:PSL3}
$G:={\mathrm{PSL}}_3(p^f)$ and $G':={\mathrm{PSL}}_3(p'^{f'})$  with $4(p^f+1)_2=4(p'^{f'}+1)_2=2^n$ and let $P\in{\mathrm{Syl}}_2(G)\cap{\mathrm{Syl}}_2(G')$.
Then, ${\mathrm{Sc}}(G\times G',\,\Delta P)$ induces 
a splendid Morita equivalence between $B_0(kG)$ and $B_0(kG')$.
\end{Proposition}

\begin{proof}
First, we claim that $M:={\mathrm{Sc}}(G\times G',\,\Delta P)$ induces a stable equivalence of Morita type between $B_0(kG)$ and $B_0(kG')$. 
Let $z$ be the unique  involution in ${Z:=Z(P)}$, and set $C:=C_G(z)$, $C':=C_{G'}(z)$, $\overline{C}:=C/O_{2'}(C)$, $\overline{C'}:=C'/O_{2'}(C')$
and $\overline{P}:= P O_{2'}(C)/O_{2'}(C) \cong P\cong
P O_{2'}(C')/O_{2'}(C')$ 
(and we identify the two groups).
Then, by \cite[Proposition 4(iii), p.21]{ABG70} and Theorem \ref{thm:semidihderal}, we obtain that 
\begin{align*} 
\overline{C}&\cong 
{\mathrm{SL}}^\pm_2(p^f) \rtimes C_d  
\ \ \ \ \text{ for an odd }d\text{ with } d\,|\,f
\text{ and }
\\
\overline{C'}&\cong 
{\mathrm{SL}}^\pm_2(p'^{f'}) \rtimes C_{d'}
\ \ \text{ for an odd }d' \text{ with } d' \,|\,f'.
\end{align*}
\noindent  We can consider that $B_0(kC)=B_0(k\overline{C})$,
$B_0(kC')=B_0(k\overline{C'})$ and
$\overline{P}
\in{\mathrm{Syl}}_2(\overline{C})\cap{\mathrm{Syl}}_2(\overline{C'})$. 
Hence it follows from Proposition~\ref{thm:SL2^pm}  and Lemma \ref{lem:AlperinDade} that 
${\mathrm{Sc}}(\overline{C}\times \overline{C'}, \Delta\overline{P})$ induces a Morita equivalence between 
$B_0(k\overline{C})$ and $B_0(k\overline{C'})$.  
Thus,
${M_Z}:={\mathrm{Sc}}(C\times C', \Delta P)$ induces a Morita equivalence between 
$B_0(kC)$ and $B_0(kC')$ {\color{black}by \cite[Proposition~3.3(b)]{KL20b}.}
On the other hand, $\mathcal F_P(G)= \mathcal F_P(G')$ by  \cite[Theorem 5.3]{CG12}.
Hence it follows from \cite[Lemma 3.2]{KL20a} that 
$M_Z\,|\,M(\Delta Z)$ and therefore by \cite[Theorem 1.2]{KT19},
$M_Z= M(\Delta Z)$.
Thus, again the gluing method of 
\cite[Lemma 4.1]{KL20a} implies the claim.
\par
Next we claim that the stable equivalence of Morita type between $B_0(kG)$ and $B_0(kG')$ 
induced by $M$ is actually a Morita equivalence.
Since ${\mathrm{Aut}}(P)$ is a $2$-group, $N_G(P)=P\times O_{2'}(C_G(P))$,
so that $N_G(P)\leq C$ and we can consider the Green correspondences 
$f:=f_{(G,P,C)}$ and $f':=f_{(G',P,C')}$. Then, it follows from \cite[(3.4)]{Erd79} that
we can consider
$$  B_0(kC)=B_0(kC') \text{ and } f(S)=f'(S')$$ 
for all three simple $kG$-modules
$S$ and $S'$ in $B_0(kG)$ and $B_0(kG')$, respectively, where $S$ corresponds to $S'$.
Thus,  
\cite[Theorem 4.14.10]{Lin18} yields that $M$ induces a Morita equivalence
between $B_0(kG)$ and \medskip $B_0(kG')$, which is automatically splendid.
\end{proof}

\begin{Proposition}\label{thm:PSU3}
Let $G:={\mathrm{PSU}}_3(p^f)$ and $G':={\mathrm{PSU}}_3(p'^{f'})$  with  $4(p^f-1)_2=4(p'^{f'}-1)_2=2^n$ 
and let $P\in{\mathrm{Syl}}_2(G)\cap{\mathrm{Syl}}_2(G')$.
Then, ${\mathrm{Sc}}(G\times G',\,\Delta P)$ induces a splendid 
Morita equivalence between $B_0(kG)$ and $B_0(kG')$.
\end{Proposition}

\begin{proof}
The same arguments as in the proof of Proposition~\ref{thm:PSL3} yield the result. More precisely, in this case $\overline{C}\cong {\mathrm{SU}}^\pm_2(p^f)$ and $\overline{C'}\cong {\mathrm{SU}}^\pm_2(p'^{f'})$
so that it follows from Proposition~\ref{thm:SU2^pm} that the Scott module 
$\overline{M_Z}:={\mathrm{Sc}}(\overline{C}\times \overline{C'}, \Delta\overline{P})$ induces a Morita equivalence between 
$B_0(k\overline{C})$ and $B_0(k\overline{C'})$, and \cite[(3.4)]{Erd79} is replaced \medskip by~\cite[(4.10)]{Erd79}. 
\end{proof}

\noindent Finally we deal with the groups of type {\sf (ab)}, that is of the form ${\mathrm{PGL}}_2^*(p^{2f})$. This case requires more involved arguments. However, the proof of \cite[Proposition~5.4]{KL20a} -- showing that the principal blocks of ${\mathrm{PGL}}_2(q)$ and ${\mathrm{PGL}}_2(q')$ with a common dihedral Sylow $2$-subgroup and $q\equiv q'\equiv 1\pmod{4}$ are splendidly Morita equivalent  -- can be imitated because the vast majority of the arguments rely on the fact that ${\mathrm{PGL}}_2(q)$ is an extension of degree two of ${\mathrm{PSL}}_2(q)$.

\begin{Proposition}\label{thm:PGL2^*}
Let $G:={\mathrm{PGL}}_2^*(p^{2f})$, $G':={\mathrm{PGL}}^*_2({p'}^{2f'})$ with $2(p^{2f}-1)_2=2(p'^{2f'}-1)_2$
$=2^n$ and let $P\in{\mathrm{Syl}}_2(G)\cap{\mathrm{Syl}}_2(G')$. 
Then, 
${\mathrm{Sc}}(G\times G',\,\Delta P)$ induces a splendid 
Morita equivalence between $B_0(kG)$ and $B(kG')$.
\end{Proposition}

\begin{proof} 
Set $B:=B_0(kG)$, $B':=B_0(kG')$
and $M:={\mathrm{Sc}}(G\times G', \Delta P)$.
By {\color{black} the definition of $G$ and $G'$ in Section~1},   
there are normal subgroups 
$N\vartriangleleft G$ and $N'\vartriangleleft G'$ with
$G$ has a normal subgroup $N$
$|G/N|=2$, $|G'/N'|=2$ and  $N\cong {\mathrm{PSL}}_2(p^{2f})$,   $N'\cong {\mathrm{PSL}}_2(p'^{2f'})$. 
Hence there is  $Q\in{\mathrm{Syl}}_2(N)\cap{\mathrm{Syl}}_2(N')$ such that  $Q\cong D_{2^{n-1}}$ (see \cite{ABG70}). 
Recall that $\mathcal F_P(G)= \mathcal F_P(G')$ 
by  \cite[Theorem~5.3]{CG12}.
\par
First, we claim that 
\begin{equation}\label{StabEquivPGL*}
M \text{ realizes a stable equivalence of Morita type between } B\text{ and }B'.
\end{equation}
Let $z$ be the unique involution in 
$Z(P)$. Set $C:=C_G(z)$ and $C':=C_{G'}(z)$.
We know that 
$z \in Q\in{\mathrm{Syl}}_2(N)\cap{\mathrm{Syl}}_2(N')$.
Now recall that $C_N(z)$ and $C_{N'}(z)$ are both
$2$-nilpotent by \cite[Lemma (7A)]{Bra66}. 
Hence, as $|G/N|=2=|G'/N'|$,
$C$ and $C'$ are also $2$-nilpotent. 
Set    
$\overline C:=C/O_{2'}(C)$, $\overline{C'}:=C'/O_{2'}(C')$, 
$\overline P:=[P\,O_{2'}(C)]/O_{2'}(C)
 \cong [P\,O_{2'}(C')]/O_{2'}(C')$.
Obviously,
$\overline{C} \cong \overline{C'}\cong \overline{P}\cong P$.
Hence, \cite[Lemma 3.1]{KL20a} implies that  ${\mathrm{Sc}}( C \times C',\Delta P)$ 
induces a Morita equivalence between $B_0(kC)$
and $B_0(kC')$ and  \cite[Lemma 3.2]{KL20a} yields
$${\mathrm{Sc}}( C \times C',\Delta P)\,\big|\, M(\Delta \langle z\rangle)$$
However, as $M$ is Brauer indecomposable by \cite{KT19}, we have  ${\mathrm{Sc}}( C \times C',\Delta P) = M(\Delta \langle z\rangle)$. 
Since, by \cite[Proposition 1.1(iii), p.10]{ABG70}, 
all involutions in $G$ are $G$-conjugate and all involutions in $G'$ are $G'$-conjugate, (\ref{StabEquivPGL*}) follows from  
\cite[Lemma 4.1]{KL20a} as we have already seen in the proof of Lemma~\ref{stableEquiv}.
\par
Second, in order to prove that the stable equivalence realized by $M$ is in fact a Morita equivalence, by \cite[Theorem 4.14.10]{Lin18} it suffices to prove that 
all simple $B$-modules are mapped to simple $B'$-modules ($\ast$).
However, to do this it is enough to note that in the statement of \cite[Proposition~5.4]{KL20a} the groups ${\mathrm{PGL}}_2({q})$ and ${\mathrm{PGL}}_2({q'})$ can be replaced with $G={\mathrm{PGL}}^*_2({p^{2f}})$ and $G'={\mathrm{PGL}}^*_2({p'^{2f'}})$ and the proof of \cite[Proposition~5.4]{KL20a} as well as the proof of the case $q\equiv 1\pmod{4}$ of \cite[Lemma~5.3]{KL20a}, on which the latter proposition relies, can both remain unchanged to yield ($\ast$). (This is because the arguments involved only rely on the facts that  ${\mathrm{PSL}}_2({q})$ is normal of index~2 in ${\mathrm{PGL}}_2({q})$ and  $\ell(B_0({\mathrm{PGL}}_2({q})))=2$, which is also true for $G$ and $G'$.)  
\end{proof}

\section{Proof of Theorem~\ref{MainResult}}

We can now prove Theorem~\ref{MainResult}.

\begin{proof}[{\bf Proof of Theorem~\ref{MainResult}}]
Part (b) is given by the case-by-case analysis of Section~\ref{sec:proofMainb}. Hence it remains to prove (a). 
\par 
To start with,  
$B_0(kG)=\mathrm{Sc}(G\times{\color{black}[G/O_{2'}(G)]},
\Delta P)$ 
(seen as a $(kG, k[G/O_{2'}(G)])$-bimodule) induces a splendid Morita equivalence between $B_0(kG)$ and $B_0(k[G/O_{2'}(G)])=:\bar{B}$, because $O_{2'}(G)$ acts trivially on the principal block.  Furthermore, if  $G'$ denotes one of the groups listed in Theorem~\ref{MainResult}(a) and there is a splendid Morita equivalence between $B_0(k[G/O_{2'}(G)])$ and $B_0(kG')$ realized by the Scott module 
$\mathrm{Sc}({\color{black}[G/O_{2'}(G)]}
\times G',\Delta P)$, then composing both equivalences, we obtain a splendid Morita equivalence between $B_0(kG)$ and $B_0(kG')$ realized by 
$$\mathrm{Sc}(G\times {\color{black}[G/O_{2'}}(G)], 
\Delta P)  \otimes_{\bar{B}} 
\mathrm{Sc}({\color{black}[G/O_{2'}(G)]} \times G',\Delta P) \cong  \mathrm{Sc}(G\times G',\Delta P) \,.$$
Therefore, we may assume that $O_{2'}(G)=1$ and so $G$ must be of type {\rm(x)}, where  {\rm(x)} denotes one  
of the seven families of groups {\rm (bb), (ba1), (ba2), (ab), (aa1), (aa2), (aa3)}  of \smallskip Theorem~\ref{thm:semidihderal}. 

\noindent\textbf{Claim 1:} $B:=B_0(kG)$ is splendidly Morita equivalent to the principal block $B':=B_0(kG')$ for some group $G'$ of type {\sf (bb), (ba1), (ba2), (ab), (aa1), (aa2)} listed in Theorem~\ref{MainResult}(a) with $P\in \Syl_2(G)\cap\Syl_2(G')$ and the splendid Morita equivalence is realized by \smallskip ${\mathrm{Sc}}(G\times G',\Delta P)$. \smallskip \\
Here we emphasize that the lists of groups in Theorem~\ref{thm:semidihderal} and in the statement of Theorem~\ref{MainResult}(a) are not the same, hence we use different fonts to distinguish them. 
We prove Claim 1 through a case-by-case analysis as follows.\smallskip \\
$\bullet$ {\sl Suppose that $G$  
is of type {\rm (bb)}.}
Then $G=P$ by Theorem \ref{thm:semidihderal}(bb). Then, we may take $G':=P$, that is $G'$ of type {\sf(bb)}. 
Obviously $B=B'=kP$ and ${\mathrm{Sc}}(P\times P,\Delta P)={_{kP}{}}kP_{kP}$ induces a splendid Morita equivalence
between $B$ and $B'$, as required. \smallskip \\
$\bullet$ {\sl Suppose that $G$ is of type {\rm(ba1)}}.
Then  
$G={\mathrm{SL}}^{\pm}_2(p^f)\rtimes C_d$ 
where $4(p^f+1)_2=2^n$ and $d$ is an odd divisor of $f$.
We take   
$G':={\mathrm{SL}}^{\pm}_2({p'}^{f'})$,  
that is of type {\sf(ba1)}, and we may assume that we have chosen $P$ such that $P\in \Syl_2(G)\cap\Syl_2(G')$. 
Then, by Frattini's argument  $G=N_G(P)G'=C_G(P)PG'=C_G(P)G'$ 
and it follows from Lemma~\ref{lem:AlperinDade}  
(i.e. \cite[Theorem~2.2(b)]{KL20a})  
that $$1_B kG 1_{B'} ={\mathrm{Sc}}(G\times G', \Delta P)$$
induces a splendid Morita equivalence between $B$ and $B'$. \smallskip \\ 
$\bullet$ {\sl Suppose that $G$  
is of type {\rm(ba2)}.} 
Then  
$G={\mathrm{SU}}^{\pm}_2(p^f)\rtimes C_d$ 
where $4(p^f-1)_2=2^n$ and $d$ is an odd divisor of $f$.
We take  
{\color{black}
$G':={\mathrm{SU}}^{\pm}_2({p'}^{f'})$, 
}
that is of type {\sf(ba2)}, and we may assume that we have chosen $P$ such that $P\in \Syl_2(G)\cap\Syl_2(G')$.
Then the same arguments as in case {\rm(ba1)} yield the claim. \smallskip \\ 
$\bullet$ {\sl Suppose that $G$  
is of type {\rm(ab)}.} 
Then 
$G={\mathrm{PGL_2^*}}(p^{2f})\rtimes C_d$
where $2(p^{2f}-1)_2=2^n$ and $d$ is an odd divisor of $f$.  
We take   
$G':={\mathrm{PGL_2^*}}({p'}^{2f'})$,
that is of type {\sf(ab)}, and we may assume that we have chosen $P$ such that $P\in \Syl_2(G)\cap\Syl_2(G')$. 
Then the same arguments as in case {\rm(ba1)} yield the claim. \smallskip\\
$\bullet$ {\sl Suppose that $G$ 
is of type {\rm(aa1)}.}
Then 
$G={\mathrm{PSL}}_3(p^f).H$ 
where $4(p^f+1)_2=2^n$ and $H \leq C_{(3, p^f-1)}\times C_d$
for an odd divisor $d$ of $f$. 
We take   
$G':={\mathrm{PSL}}_3({p'}^{f'})$, 
that is of type {\sf(aa1)}, and we may assume that we have chosen $P$ such that $P\in \Syl_2(G)\cap\Syl_2(G')$. 
Then the same arguments as in case {\rm(ba1)} yield the claim, where $C_d$ is replaced by $H$. \smallskip\\
$\bullet$ {\sl Suppose that $G$  
is of type {\rm(aa2)}.}
Then  
$G={\mathrm{PSU}}_3(p^f).H$ 
where $4(p^f-1)_2=2^n$ and $H \leq C_{(3, p^f+1)}\times C_d$
for an odd divisor $d$ of $f$.
We take  
$G':={\mathrm{PSU}}_3({p'}^{f'})$,   
that is of type {\sf(aa2)}, and we may assume that we have chosen $P$ such that $P\in \Syl_2(G)\cap\Syl_2(G')$.
Then the same arguments as in case {\rm(ba1)} yield the claim, where $C_d$ is replaced by $H$. \smallskip\\
$\bullet$ {\sl Suppose that $G$  
is of type {\rm(aa3)}.}
Then $G=M_{11}$ by Theorem~\ref{thm:semidihderal}(aa3) and $n=4$. We take $G':=\mathrm{PSL}_3(3)$, so that $P\in{\mathrm{Syl}}_2(G)\cap{\mathrm{Syl}}_2(G')$ and $G'$ is of type {\sf(aa1)}. 
Moreover, by Proposition~\ref{M11-L3(3)}, ${\mathrm{Sc}}(G\times G',\Delta P)$ induces a splendid Morita equivalence
between $B$ and $B'$, as required. \smallskip \\
Furthermore, the fact that the group $G'$ in Claim 1 is independent of the choice of $p$ and~$f$ for types {\sf (ba1), (ba2), (ab), (aa1), (aa2)} follows directly from Part (b). Hence it only remains to prove the following claim.\smallskip \\
\noindent\textbf{Claim 2.} The principal blocks of the groups listed 
in the different cases of Theorem~\ref{MainResult}(a)
are mutually not splendidly
 Morita \smallskip equivalent.\\
So let  $B:=B_0(kG)$ 
for $G$ of type {\sf(x)} with 
{\sf(x)}$\in\{${\sf(bb), (ba1), (ba2), (ab), (aa1), (aa2)}$\}$ as in Theorem~\ref{MainResult}(a). It is enough to show  that $B$ is not Morita equivalent to $B':=B_0(kG')$ for $G'$ of type 
{\sf(y)} $ \neq \, ${\sf(x)} and {\sf(y)}$\in\{${\sf(bb), (ba1), (ba2), (ab), (aa1), (aa2)}$\}$. 
\par
Now, type {\sf(bb)} is the unique case in which $\ell(B)=1$, so we can assume that $G$ is not of type {\sf(bb)}. Next, assume that $\ell(B)=3$. Then, by \cite[table on p. 231]{Ols75},
$G$ is of type {\sf(aa1)} or {\sf(aa2)}.
However, the principal blocks of groups of type {\sf(aa1)} and {\sf(aa2)} are never 
 Morita equivalent
because their $2$-decomposition matrices are different by \cite[SD$(2\mathcal B)_2$, p.299 and  SD$(2\mathcal A)_1$, p.298]{Erd90}.
Therefore, it only remains to consider the case $\ell(B)=2$. Then, by \cite[table on p. 231]{Ols75}, $G$ is of type {\sf(ab)} or {\sf(ba)} (i.e. {\sf(ba1)} or {\sf(ba2)}).  
Then, by looking at  
$k(B)$s (see \cite[the table on p.231]{Ols75}), we obtain that the principal blocks of groups of type {\sf(ab)} and {\sf(ba)} are never  Morita equivalent. 
Hence we can also assume that $G$ is not of type {\sf(ab)}, so that we may assume that $G$ is of type {\sf(ba1)} and $G'$ is of type {\sf(ba2)}, that is 
$$
\begin{matrix}
&
&
&G={\color{black}{\mathrm{SL}}^\pm_2(p^f) }
&\text{with}
&4(p^f+1)_2
&=2^n
\\
&
&
&\ \ G'={\color{black}{\mathrm{SU}}^\pm_2({p'}^{f'}) }
&\text{with} 
&\ 4({p'}^{f'}-1)_2
&=2^n
\end{matrix}
$$ 
and we may identify a Sylow $2$-subgroup $P$ of $G$ and $G'$.
Since $G$ has a central involution, say $z$,  
set $Z:=\langle z\rangle$, $\overline{G}:=G/Z{\color{black} \cong {\mathrm{PGL}}_2(p^f)}$,
$\overline{B}:= B_0(k \overline{G})$, $\overline{P}:=P/Z$ and note that $\overline{P}\cong D_{2^{n-1}}$.
We also have $Z\leq G'$, hence we can set 
$\overline{G'}:=G'/Z {\color{black} \cong {\mathrm{PGL}}_2(p'^{f'})}$ and
$\overline{B'}:= B_0(k\overline{G'})$. 
Then, it follows from the condition on ${p}^{f}$ and 
\cite[Corollary 8.1(f)]{KL20a}
(see also \cite[D(2$\mathcal B$), p.295]{Erd90})
and from the condition on ${p'}^{f'}$
and \cite[Corollary~8.1(e)]{KL20a}
(see \cite[D(2$\mathcal A$), p.294]{Erd90}),
respectively, that
$$C_{\overline{B} } =
\begin{pmatrix} 4& 2\\ 2& 2^{n-3}+1 \end{pmatrix}
\text{\quad and \quad}
C_{\overline{B'}} =
\begin{pmatrix} 2^{n-1}& 2^{n-2}\\ 2^{n-2}& 2^{n-3}+1 \end{pmatrix}.
$$
Now suppose that $B$ and $B'$ are 
Morita equivalent.
Then, $C_B=C_{B'}$. 
Since $Z$ is a central subgroup of $G$ and $G'$ of order $2$, \cite[Theorem 5.8.11]{NT89} 
implies that
$C_B=2C_{\overline{B}}$ and
$C_{B'}= 2C_{\overline{B'}}.$
Thus $C_{\overline{B}}=C_{\overline{B'}}$,
a contradiction. Claim 2 follows. 
\end{proof}

\bigskip

\noindent
{\bf Acknowledgment.}
{\small The authors are grateful to Burkhard K\"{u}lshammer for
useful conversations.



\begin{thebibliography}{CEKL11}


{\color{black}
\bibitem[Alp76]{Alp76}
J.L.~Alperin,
\emph{Isomorphic blocks}, J.~Algebra \textbf{43} (1976), 694--698.
}

\bibitem[ABG70]{ABG70}
{\sc J.L.~Alperin, R.~Brauer, D.~Gorenstein},
\emph{Finite groups with quasi-dihedral and wreathed Sylow $2$-subgroups},
Trans.~Amer.~Math.~Soc. \textbf{151} (1970), 1--261.

\bibitem[BC87]{BC87}
{\sc D.J.~Benson, J.F.~Carlson},
\emph{Diagrammatic methods for modular representations and 
\linebreak
cohomology},
Comm.~Algebra \textbf{15} (1987), 53--121.

{\color{black}
\bibitem[Bra66]{Bra66}
{\sc R.~Brauer},
\emph{Some applications of the theory of blocks of characters of finite groups III}, 
J.~Algebra~\textbf{3} (1966), 225--255.
}



\bibitem[Atlas]{Atlas}
{\sc J.H.~Conway, R.T.~Curtis, S.P.~Norton, R.A.~Parker, R.A.~Wilson}, 
\emph{Atlas of Finite Groups}, Clarendon Press, Oxford, 1985.



\bibitem[CEKL11]{CEKL11}
{\sc D.A.~Craven, C.W.~Eaton, R.~Kessar, M.~Linckelmann},
\emph{The structure of blocks with a Klein four defect group}, 
Math.~Z. \textbf{268} (2011), 441--476.

\bibitem[CG12]{CG12}
{\sc D.A.~Craven, A. Glesser},
\emph{Fusion systems on small $p$-groups},
Trans.~Amer.~Math.~Soc. \textbf{364} (2012), 5945--5967.

{\color{black}
\bibitem[Dad77]{Dad77}
E.C.~Dade, \emph{Remarks on isomorphic blocks},
J.~Algebra \textbf{45} (1977), 254--258. 
}

\bibitem[Erd79]{Erd79}
{\sc K.~Erdmann},
\emph{On $2$-blocks with semidihedral defect groups},
Trans.~Amer.~Math.~Soc. \textbf{256} (1979), 267--287.

\bibitem[Erd90]{Erd90} 
{\sc K. Erdmann},
\emph{Blocks of Tame Representation Type and Related Algebras}. 
Lecture Notes in Mathematics, \textbf{vol. 1428}, Springer-Verlag, Berlin, 1990.

\bibitem[GAP]{GAP}
{\sc The GAP Group},
{\sf GAP} --- Groups, Algorithms, and Programming,
Version 4.8.4, 
\url{http:// www.gap-system.org}, 2016.

\bibitem[Go69]{G}
D. Gorenstein, \textit{Finite groups the centralizers of whose involutions have normal 2-complements}, Canad.~ J. Math. \textbf{21} (1969), 335--357. 

\bibitem[HM76]{HM76}
{\sc W.~Hamernik, G.O.~Michler}, 
\emph{On vertices of simple modules in $p$-solvable groups}.
Mitt.~Math.~Sem.~Giessen {\bf 121} (1976), 147–162.

\bibitem[HK00]{HK00}
{\sc G. Hiss, R. Kessar},
\emph{Scopes reduction and Morita equivalence classes of blocks 
in finite classical groups}, 
J.~Algebra \textbf{230} (2000), 378--423.

\bibitem[HK05]{HK05}
{\sc G. Hiss, R. Kessar},
\emph{Scopes reduction and Morita equivalence classes of blocks 
in finite classical groups II},  
J.~Algebra \textbf{283} (2005), 522--563.

\bibitem[KL20a]{KL20a}
{\sc S. Koshitani, C. Lassueur},
\emph
{Splendid Morita equivalences for principal $2$-blocks with dihedral defect groups}, 
Math.~Z. \textbf{294} (2020), 639--666.

\bibitem[KL20b]{KL20b}
{\sc S. Koshitani, C. Lassueur},
\emph
{Splendid Morita equivalences for principal $2$-blocks with generalised quaternion defect groups}, 
J.~Algebra \textbf{558} (2020), 523--533.

\bibitem[KT19]{KT19}
{\sc S.~Koshitani, \.I.~Tuvay},
\emph{The Brauer indecomposability of Scott modules with semidihedral vertex},
preprint, \href{https://arxiv.org/pdf/1908.05536v2.pdf}{arXiv:1908.05536v2}

\bibitem[Lan83]{Lan83}
{\sc P.~Landrock},
\emph{Finite Group Algebras and their Modules},
London Math.~Soc.~Lecture Note Series, \textbf{vol.84}, 
Cambridge Univ.~Press, Cambridge, 1983.


\bibitem[Lin96]{Lin96b}
{\sc M. Linckelmann},
\emph{The isomorphism problem for cyclic blocks and their source algebras}, \linebreak
Invent.~Math. \textbf{125} (1996), 265--283.

\bibitem[Lin01]{Lin01}
{\sc M.~Linckelmann},
\emph{On splendid derived and stable equivalences between blocks of finite groups},
J.~Algebra \textbf{242} (2001), 819--843.

\bibitem[Lin18]{Lin18}
{\sc M.~Linckelmann},
\emph{The Block Theory of Finite Group Algebras, Volumes 1 and 2},
London Math.~Soc.~Student Texts \textbf{91} and \textbf{92}, 
Cambridge Univ.~Press, Cambridge, 2018, 

\bibitem[NT88]{NT89}  
{\sc H.~Nagao, Y. Tsushima}, 
\emph{Representations of Finite Groups},
Academic Press, New York, 1988.

\bibitem[Ols75]{Ols75}
{\sc J.B.~Olsson},
\emph{On $2$-blocks with quaternion and quasidihedral defect groups},
J.~Algebra \textbf{36} (1975), 212--241.

\bibitem[Pui94]{Pui94} 
{\sc L. Puig},
\emph{On Joanna Scopes' criterion of equivalence for blocks of symmetric groups},
Algebra Colloq. \textbf{1} (1994), 25--55.

\bibitem[Pui99]{Pui99}
{\sc L. Puig},
\emph{On the Local Structure of Morita and Rickard Equivalences 
between Brauer Blocks}, 
Birkh{\"a}user, Basel, 1999.

\bibitem[Sch83]{SchM12} 
{\sc G. J.A.~Schneider},
\emph{The vertices of the simple modules of $M_{12}$ over a field of characteristic~$2$},
J.~Algebra \textbf{83} (1983), 189--200.

\bibitem[Th\'e85]{The85} 
{\sc J.~Th{\'e}venaz},
\emph{Relative projective covers and almost split sequences},
Comm.~Algebra \textbf{13} (1985), 1535--1554.

\bibitem[Th\'e95]{The95} 
{\sc J.~Th\'{e}venaz},
\emph{$G$-Algebras and Modular Representation Theory}. 
Clarendon Press, Oxford, 1995.

\bibitem[ModAtl]{ModAtlas}
{\sc R.~Wilson, J.~Thackray, R.~Parker, F.~Noeske,
J.~M{\"u}ller, F.~L{\"u}beck, C.~Jansen, G.~Hiss, T.~Breuer},
\emph{The Modular Atlas Project},
\url{http://www.math.rwth-aachen.de/~MOC}.

\end{thebibliography}
\end{document}